\documentclass[a4paper,12pt]{amsart}

\usepackage{a4wide}
\usepackage{pgf,tikz}
\usepackage{amssymb}
\usepackage{enumerate}

\usepackage{amsmath}

\newif\ifdetails
\detailstrue
\newcommand{\DETAIL}[1]%
{\ifdetails\par\fbox{\begin{minipage}{0.9\linewidth}\textit{Detail:}
      #1\end{minipage}}\par\fi}
\newcommand{\TODO}[1]%
{\ifdetails\par\fbox{\begin{minipage}{0.9\linewidth}\textbf{TODO:}
      #1\end{minipage}}\par\fi}

\usepackage{makecell}
\usepackage{relsize}

\newtheorem{lemma}{Lemma}
\newtheorem{proposition}[lemma]{Proposition}
\newtheorem{theorem}[lemma]{Theorem}

\theoremstyle{remark}

\usepackage{caption}
\usepackage{subcaption}
\usepackage{float} 
\usepackage[pdftex,a4paper,
citecolor = blue, colorlinks=true,urlcolor=blue]{hyperref}
\usepackage{mathrsfs}
\usetikzlibrary{arrows}

\newcommand{\old}[1]{{}}

\title{Leaf-induced subtrees of leaf-Fibonacci trees}

\author{Audace A. V. Dossou-Olory}
\thanks{The author was supported by Stellenbosch University in association with African Institute for Mathematical Science (AIMS) South Africa.}
\address{Audace A. V. Dossou-Olory \\ Department of Mathematical Sciences \\ Stellenbosch University \\ Private Bag X1, Matieland 7602 \\ South Africa}
\email{audace@aims.ac.za}

\begin{document}

\begin{abstract}
In analogy to a concept of Fibonacci trees, we define the leaf-Fibonacci tree of size $n$ and investigate its number of nonisomorphic leaf-induced subtrees. Denote by $f_0$ the one vertex tree and $f_1$ the tree that consists of a root with two leaves attached to it; the leaf-Fibonacci tree $f_n$ of size $n\geq 2$ is the binary tree whose branches are $f_{n-1}$ and $f_{n-2}$. We derive a nonlinear difference equation for the number $\text{N}(f_n)$ of nonisomorphic leaf-induced subtrees (subtrees induced by leaves) of $f_n$, and also prove that $\text{N}(f_n)$ is asymptotic to $1.00001887227319\ldots (1.48369689570172 \ldots)^{\phi^n}$ ($\phi=$~golden ratio) as $n$ grows to infinity.
\end{abstract}

\maketitle

\section{Introduction}

Fibonacci trees are an alternative approach to a binary search in computer science and information processing~\cite[p.~417]{knuth1997art}. The Fibonacci tree of order $n$ is defined as the binary tree whose left branch is the Fibonacci tree of order $n-1$ and right branch is the Fibonacci tree of order $n-2$, while the Fibonacci tree of order $0$ or $1$ is the tree with only one vertex~\cite{knuth1997art}. We show in Figure~\ref{Fft5} the Fibonacci tree of order $5$.
\begin{figure}[H]\centering
\begin{tikzpicture}[thick, level distance=8mm]
\tikzstyle{level 1}=[sibling distance=36mm]
\tikzstyle{level 2}=[sibling distance=18mm]
\tikzstyle{level 3}=[sibling distance=10mm]
\tikzstyle{level 4}=[sibling distance=6mm]
\node [circle,draw]{}
child {child {child {child {[fill] circle (2pt)}child {[fill] circle (2pt)}}child {[fill] circle (2pt)}}child {child {[fill] circle (2pt)}child {[fill] circle (2pt)}}}
child {child {child {[fill] circle (2pt)}child {[fill] circle (2pt)}}child {[fill] circle (2pt)}};
\end{tikzpicture}
\caption{The Fibonacci tree of order $5$.}\label{Fft5}
\end{figure}

Thus, the Fibonacci tree of order $n$ has precisely $F_{n+1}$ leaves (so $2F_{n+1} -1$ vertices), where $F_n$ denotes the $n$-th Fibonacci number:
\begin{align*}
F_0=0,~F_1=1,~F_n=F_{n-1}+F_{n-2}~~\text{for}~~n>1\,.
\end{align*}

Fibonacci trees are also a special case of so-called AVL (``Adel'son-Vel'skii and Landis''---named after the inventors) trees~\cite{adelsonalgorithmfor}; these trees have the defining property that for every internal vertex $v$, the heights (i.e., the greatest distance of a leaf from the root) of the left and right branches of the subtree rooted at $v$ (consisting of $v$ and all its descendants) differ by at most one. According to~\cite{adelsonalgorithmfor}, AVL trees are the first data structure to be invented. Figure~\ref{AVLh3} shows an AVL tree of height $3$. For more information on Fibonacci trees and their uses, we refer to~\cite{stevens1974patterns,horibe1982entropy,horibe1983notes,grimaldi1991properties}.
\begin{figure}[H]\centering
\begin{tikzpicture}[thick, level distance=9mm]
\tikzstyle{level 1}=[sibling distance=21mm]
\tikzstyle{level 2}=[sibling distance=10mm]
\tikzstyle{level 3}=[sibling distance=6mm]
\node [circle,draw]{}
child {child {[fill] circle (2pt)} child {edge from parent[draw=none]}}
child {child {child {edge from parent[draw=none]} child {[fill] circle (2pt)}} child {[fill] circle (2pt)}};
\end{tikzpicture}
\caption{An AVL tree of height $3$.}\label{AVLh3}
\end{figure}
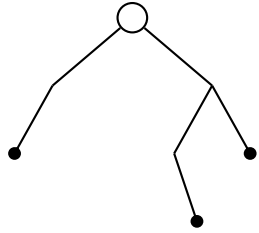

\medskip
In analogy to the concept of Fibonacci trees from~\cite{knuth1997art}, we define the \textit{leaf-Fibonacci tree of size (height) $n$} as follows:
\begin{itemize}
\item Denote by $f_0$ the tree with only one vertex and $f_1$ the tree that consists of a root with two leaves attached to it;
\item For $n\geq 2$, connect the roots of the trees $f_{n-1}$ and $f_{n-2}$ to a new common vertex to obtain the tree $f_n$.
\end{itemize}
In other words, the leaf-Fibonacci tree $f_n$ of size $n\geq 2$ is the binary tree  whose branches are the leaf-Fibonacci trees $f_{n-1}$ and $f_{n-2}$. Hence, $f_n$ has precisely $F_{n+2}$ leaves, where $F_n$ is the $n$-th Fibonacci number ($F_1=1,F_2=1,F_3=2,F_4=3,F_5=5,F_6=8,F_7=13,\ldots$). Figure~\ref{Lft5} shows the leaf-Fibonacci tree of size $5$.
\begin{figure}[H]\centering
\begin{tikzpicture}[thick, level distance=8mm]
\tikzstyle{level 1}=[sibling distance=50mm]
\tikzstyle{level 2}=[sibling distance=25mm]
\tikzstyle{level 3}=[sibling distance=13mm]
\tikzstyle{level 4}=[sibling distance=6mm]
\tikzstyle{level 5}=[sibling distance=5mm]
\node [circle,draw]{}
child {child {child {[fill] circle (2pt)}child {[fill] circle (2pt)}}child {child {[fill] circle (2pt)}child {child {[fill] circle (2pt)}child {[fill] circle (2pt)}}}}
child {child {child {[fill] circle (2pt)}child {child {[fill] circle (2pt)}child {[fill] circle (2pt)}}}child {child {child {[fill] circle (2pt)}child {[fill] circle (2pt)}}child {child {[fill] circle (2pt)}child {child {[fill] circle (2pt)}child {[fill] circle (2pt)}}}}};
\end{tikzpicture}
\caption{The leaf-Fibonacci tree of size $5$.}\label{Lft5}
\end{figure}
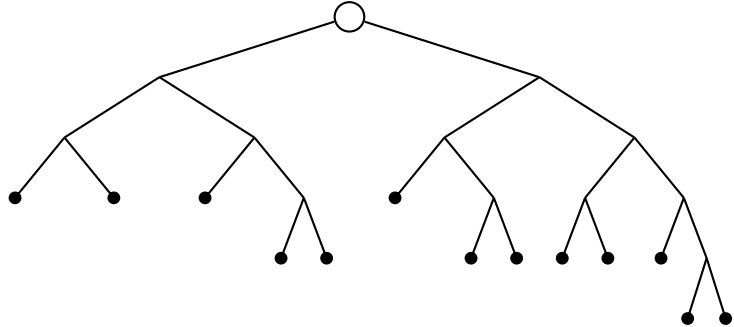

In this note, we shall be interested in the number of nonisomorphic subtrees induced by leaves (henceforth, \textit{leaf-induced subtrees}) of a leaf-Fibonacci tree of size $n$. 

\medskip
Let $T$ be a rooted tree without vertices of outdegree $1$ (also known as \textit{topological} or \textit{series-reduced} or \textit{homeomorphically irreducible} trees~\cite{bergeron1998combinatorial,allman2004mathematical,dossou2018inducibility,oeis}). Every choice of $k$ leaves of $T$ induces another topological tree, which is obtained by extracting the minimal subtree of $T$ that contains all the $k$ leaves and suppressing (if any) all vertices of outdegree $1$; see Figure~\ref{leaf-induced} for an illustration. 
\begin{figure}[htbp]\centering  
\begin{tikzpicture}[thick]

\node [circle,draw] (rt) at (-2.5,1.5) {};
\draw (rt) -- (-3.5,-4);
\draw (rt) -- (0,0);

\draw (0,0) -- (-2,-2);
\draw (0,0) -- (0,-4);
\draw (0,0) -- (2,-2);
\draw (-2,-2) -- (-2.5,-4);
\draw (-2,-2) -- (-1.5,-4);
\draw (2,-2) -- (1,-4);
\draw (2,-2) -- (2,-4);
\draw (2,-2) -- (3,-4);
\draw (1.25,-3.5) -- (1.5,-4);

\node [fill,rectangle, inner sep = 3.5pt ] at (0,0) {};
\node [fill,circle, inner sep = 2pt ] at (-3.5,-4) {};
\node [fill,circle, inner sep = 2pt ] at (-2.5,-4) {};
\node [fill,circle, inner sep = 2pt ] at (-1.5,-4) {};
\node [fill,circle, inner sep = 2pt ] at (0,-4) {};
\node [fill,circle, inner sep = 2pt ] at (1,-4) {};
\node [fill,circle, inner sep = 2pt ] at (1.5,-4) {};
\node [fill,circle, inner sep = 2pt ] at (2,-4) {};
\node [fill,circle, inner sep = 2pt ] at (3,-4) {};

\node at (-2.5,-4.5) {$\ell_1$};
\node at (0,-4.5) {$\ell_2$};
\node at (1,-4.5) {$\ell_3$};
\node at (2,-4.5) {$\ell_4$};

\node [fill,rectangle, inner sep = 3.5pt, draw] (r1) at (6,-2) {};

\draw (r1) -- (5,-4);
\draw (r1) -- (6,-4);
\draw (r1) -- (7,-4);
\draw (6.75,-3.5)--(6.5,-4);

\node [fill,circle, inner sep = 2pt ] at (5,-4) {};
\node [fill,circle, inner sep = 2pt ] at (6,-4) {};
\node [fill,circle, inner sep = 2pt ] at (6.5,-4) {};
\node [fill,circle, inner sep = 2pt ] at (7,-4) {};

\node at (5,-4.5) {$\ell_1$};
\node at (6,-4.5) {$\ell_2$};
\node at (6.5,-4.5) {$\ell_3$};
\node at (7,-4.5) {$\ell_4$};
\end{tikzpicture}
\caption{A topological tree (on the left) and the subtree induced by the leaves $\ell_1,\ell_2,\ell_3,\ell_4$ (on the right).}\label{leaf-induced}
\end{figure}
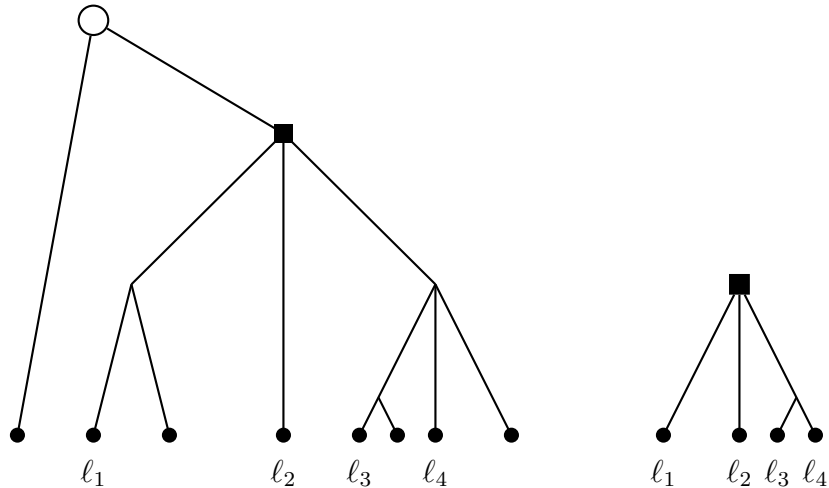
Every subtree obtained through this operation is sometimes referred to as a \textit{leaf-induced subtree}~\cite{czabarka2018inducibility,dossou2018inducibility}. We note the study of subtrees induced by leaves of binary trees finds a noteworthy relevance in phylogenetics---see Semple and Steel's book~\cite{semple2003phylogenetics} which describes the mathematical foundations of phylogenetics.

\medskip
Two rooted trees are said to be isomorphic if there is a graph isomorphism (preserving adjacency) between them that maps the root of one to the root of the other. It is important to note that the problem of enumerating leaf-induced subtrees becomes trivial if isomorphisms are not taken into account: in fact, it is clear that every topological tree with $n$ leaves has exactly $2^n -1$ leaf-induced subtrees.

We mention that nonisomorphic leaf-induced subtrees of a topological tree have been studied only very recently: Wagner and the present author~\cite{leafSubtree} obtained exact and asymptotic enumeration results on the number of nonisomorphic leaf-induced subtrees of two classes of $d$-ary trees, namely so-called $d$-ary caterpillars and even $d$-ary trees. In~\cite{leafSubtree}, they also derived extremal results for the number of root containing leaf-induced subtrees of a topological tree.

\medskip
We shall denote the number of nonisomorphic leaf-induced subtrees of the leaf-Fibonacci tree $f_n$ by $\text{N}(f_n)$. Our main results are a recurrence relation and an asymptotic formula for $\text{N}(f_n)$. As it turns out, the plan to compute $\text{N}(f_n)$ will be to consider the number of root containing leaf-induced subtrees of $f_n$.

\medskip
In~\cite{wagner2007fibonacci}, Wagner alone studied the number of independent vertex subsets (set of vertices containing no pair of adjacent vertices) of a Fibonacci tree of order $n$ with the notable difference that in his context, the Fibonacci tree of order $0$ has no vertices. Wagner derived a system of recurrence relations for the number of independent vertex subsets of a Fibonacci tree of an arbitrary order $n$, and also proved that there are positive constants $A,B>0$ such that the number of independent vertex subsets of a Fibonacci tree of order $n$ is asymptotic to $A\cdot B^{F_n}$ as $n$ grows to infinity. In the present study, we obtain a similar asymptotic formula for the number $\text{N}(f_n)$ of nonisomorphic leaf-induced subtrees of the leaf-Fibonacci tree of size $n$: we demonstrate---not expectedly---that for some effectively computable constants $A_1,A_2>0$,
\begin{align*}
\text{N}(f_n) \thicksim A_1 \cdot A_2^{F_n}~~\text{as}~~n \to \infty\,.
\end{align*}

\section{Main results} 

We note from the recursive definition of the tree $f_n$ that $f_m$ is a leaf-induced subtree of $f_n$ for every $m \leq n$. However, not every leaf-induced subtree of $f_n$ is again a leaf-Fibonacci tree: in fact, by repeatedly removing leaves from $f_n$, one easily sees that $f_n$ has leaf-induced subtrees of every number of leaves $k$ between $1$ and $n$.

As mentionned in the introduction, the plan to compute $\text{N}(f_n)$ will be to consider the number of root containing leaf-induced subtrees of $f_n$.
\begin{lemma}\label{UsefulLem}
All nonisomorphic leaf-induced subtrees with two or more leaves of $f_n$ can be identified as containing the root of $f_n$.
\end{lemma}

\begin{proof}
The tree $f_0$ has only one vertex which is also its leaf and root, so the statement holds vacuously for $n=0$. The statement is trivial for $n=1$ ($f_1$ is the only leaf-induced subtree in this case). Let $n>1$ and consider a subset of $k>1$ leaves of $f_n$. We argue by double induction on $n$ and $k$:
\begin{itemize}
\item If all $k$ leaves belong to $f_{n-1}$ then by the induction hypothesis on $n$, the induced subtree with $k$ leaves contains the root of $f_{n-1}$. Moreover, by the induction hypothesis on $k$, the tree $f_{n-1}$ can be identified as containing the root of $f_n$ (as $f_{n-1}$ is clearly a leaf-induced subtree of $f_n$). Hence, the induced subtree with $k$ leaves contains the root of $f_n$.
\item If all $k$ leaves belong to $f_{n-2}$, then we also deduce by the induction hypothesis that the induced subtree with $k$ leaves is a root containing leaf-induced subtree of $f_n$.
\item If $k_1$ leaves belong to $f_{n-1}$ and $k-k_1$ leaves belong to $f_{n-2}$, then by the induction hypothesis, the induced subtrees with $k_1$ and $k-k_1$ leaves are root containing leaf-induced subtrees of $f_{n-1}$ and $f_{n-2}$, respectively. Consequently, the root of the induced subtree with $k$ leaves coincides with the root of $f_n$.
\end{itemize}
This completes the induction step as well as the proof of the lemma. 
\end{proof}

\medskip
We then obtain the following proposition:

\begin{proposition}
The number $\text{N}(f_n)$ of nonisomorphic leaf-induced subtrees of the leaf-Fibonacci tree $f_n$ satisfies the following nonlinear recurrence relation:
\begin{align}\label{EqLFibTree}
\text{N}(f_n)=1 + \frac{1}{2}\text{N}(f_{n-2}) -\frac{1}{2}\text{N}(f_{n-2})^2 + \text{N}(f_{n-2}) \cdot \text{N}(f_{n-1}) 
\end{align}
with initial values $\text{N}(f_0)=1,\text{N}(f_1)=2$.
\end{proposition}

\begin{proof}
It is obvious that $\text{N}(f_0)=1$ and $\text{N}(f_1)=2$. Let $n>1$. By Lemma~\ref{UsefulLem}, $\text{N}(f_n)$ is precisely one more the number of nonisomorphic root containing leaf-induced subtrees of $f_n$ (the subtree with only one vertex has been included). Since all leaf-induced subtrees of the leaf-Fibonacci tree $f_{n-2}$ are again leaf-induced subtrees of $f_{n-1}$, the nonisomorphic root containing leaf-induced subtrees of $f_n$ can be categorised by two types of enumeration:
\begin{itemize}
\item Both branches of the induced subtree are leaf-induced subtrees of $f_{n-2}$. The total number of these possibilities is $\binom{1+\text{N}(f_{n-2})}{2}$ as the induced subtrees have to be nonisomorphic.
\item One of the branches of the induced subtree is a leaf-induced subtree of $f_{n-2}$ while the other branch is a leaf-induced subtree of $f_{n-1}$ but does not belong to the set of leaf-induced subtrees of $f_{n-2}$. The total number of these possibilities is $\text{N}(f_{n-2}) (\text{N}(f_{n-1}) - \text{N}(f_{n-2}))$.
\end{itemize}
Therefore, we obtain
\begin{align*}
\text{N}(f_n)&=1 + \binom{1+\text{N}(f_{n-2})}{2} + \text{N}(f_{n-2}) (\text{N}(f_{n-1}) - \text{N}(f_{n-2}))\\
&= 1 + \frac{1}{2}\text{N}(f_{n-2}) -\frac{1}{2}\text{N}(f_{n-2})^2 + \text{N}(f_{n-2}) \cdot \text{N}(f_{n-1})\,,
\end{align*}
which completes the proof of the proposition.
\end{proof}

\medskip
The sequence $(\text{N}(f_n))_{n\geq 0}$ starts as
\begin{align*}
&\text{N}(f_0)=1,~\text{N}(f_1)=2,~\text{N}(f_2)=3,~\text{N}(f_3)=6,~\text{N}(f_4)=16,~\text{N}(f_5)=82,~\text{N}(f_6)=1193,\\
&\text{N}(f_7)=94506,~\text{N}(f_8)=112034631,\ldots
\end{align*}

\medskip
We remark that recursion~\eqref{EqLFibTree} cannot be solved explicitly. Therefore, finding an asymptotic formula should be in order. In the following theorem, we show---not expectedly---that $\text{N}(f_n)$ grows doubly exponentially in $n$.

\begin{theorem}
There are two positive constants $K_1,K_2>0$ (both solely depending on the first terms of $(\text{N}(f_n))_{n\geq 0}$) such that
\begin{align*}
\text{N}(f_n) =(1+o(1)) K_1 \cdot K_2^{\big(\frac{1+\sqrt{5}}{2}\big)^n}
\end{align*}
as $n\to \infty$.
\end{theorem}

\begin{proof}
For ease of notation, set $A_n:=\text{N}(f_n)$. Then we have
\begin{align*}
A_n=1 + \frac{1}{2}A_{n-2} -\frac{1}{2}A_{n-2}^2 + A_{n-2} \cdot  A_{n-1}
\end{align*}
with initial values $A_0=1,A_1=2$. Since the sequence $(A_n)_{n\geq 0}$ increases with $n$, it is not difficult to note that
\begin{align*}
A_n \geq \frac{1}{2}A_{n-1} \cdot A_{n-2}
\end{align*}
for all $n\geq 2$. Also, since $A_n\geq A_2=3$ for all $n\geq 2$ and $1+A_1/2-A_1^2/2=0$, it is not difficult to see that
\begin{align*}
A_n \leq A_{n-1} \cdot A_{n-2}
\end{align*}
for all $n\geq 3$. Thus, we have
\begin{align}\label{IntExp}
\lim_{n\to \infty} \frac{A_{n-1}}{A_n}=0\,,
\end{align}
which also implies that the sequence $(A_{n-1}/A_n)_{n\geq 1}$ is bounded for every $n\geq 1$. Let us use $Q_n$ as a shorthand for $\log(A_n)$ and $E_n$ as a shorthand for
\begin{align}\label{EnEp}
\log\Big(1+\frac{1}{2A_{n-1}} -\frac{A_{n-2}}{2A_{n-1}}+\frac{1}{A_{n-1}\cdot A_{n-2}}\Big)\,.
\end{align}
With these notations, we have
\begin{align*}
Q_n=Q_{n-1}+Q_{n-2}+E_n\,.
\end{align*}
By setting $R_{n-1}:=Q_{n-2}$, we obtain the following system (written in matrix form) of two linear difference equations:
\begin{align*}
\left(
\begin{array}{c}
Q_n \\ R_n
\end{array}
\right)=
\left(
\begin{array}{cc}
1 & 1 \\
1 & 0
\end{array}
\right)
\left(
\begin{array}{c}
Q_{n-1} \\ R_{n-1}
\end{array}
\right)+
\left(
\begin{array}{c}
E_n \\ 0
\end{array}
\right)\,.
\end{align*}
By iteration on $n$, one gets
\begin{align*}
\left(
\begin{array}{c}
Q_n \\ R_n
\end{array}
\right)=
\left(
\begin{array}{cc}
1 & 1 \\
1 & 0
\end{array}
\right)^{n-1}
\left(
\begin{array}{c}
Q_1 \\ Q_0
\end{array}
\right)+
\sum_{i=2}^n \left(
\begin{array}{cc}
1 & 1 \\
1 & 0
\end{array}
\right)^{n-i}
\left(
\begin{array}{c}
E_i \\ 0
\end{array}
\right)
\end{align*}
for all $n\geq 2$, as $R_1=Q_0$. The eigenvalue decomposition gives us
\begin{align*}
\left(\begin{array}{cc}
1 & 1 \\
1 & 0
\end{array}\right)=\frac{1}{\lambda_1 -\lambda_2}\left(
\begin{array}{cc}
\lambda_1 & \lambda_2 \\
1 & 1
\end{array}\right) \left(\begin{array}{cc}
\lambda_1 & 0 \\
0 & \lambda_2
\end{array}\right) \left(
\begin{array}{cc}
1 & -\lambda_2 \\
-1 & \lambda_1
\end{array}\right)
\end{align*}
with
\begin{align*}
\lambda_1=\frac{1-\sqrt{5}}{2} \quad \text{and} \quad \lambda_2=\frac{1+\sqrt{5}}{2}\,.
\end{align*}

It follows that
\begin{align*}
\left(\begin{array}{cc}
1 & 1 \\
1 & 0
\end{array}\right)^m&=\frac{1}{\lambda_1 -\lambda_2}\left(
\begin{array}{cc}
\lambda_1 & \lambda_2 \\
1 & 1
\end{array}\right) \left(\begin{array}{cc}
\lambda_1^m & 0 \\
0 & \lambda_2^m
\end{array}\right) \left(
\begin{array}{cc}
1 & -\lambda_2 \\
-1 & \lambda_1
\end{array}\right)\\
&=\frac{1}{\lambda_1 -\lambda_2} \left(
\begin{array}{cc}
\lambda_1^{m+1}-\lambda_2^{m+1}  & \lambda_1\cdot \lambda_2^{m+1}-\lambda_1^{m+1}\cdot \lambda_2 \\
\lambda_1^m-\lambda_2^m  & \lambda_1\cdot \lambda_2^m-\lambda_1^m\cdot \lambda_2
\end{array}\right) 
\end{align*}
for all integer values of $m$. Consequently, we have
\begin{align*}
\left(
\begin{array}{c}
Q_n \\ R_n
\end{array}
\right)=\frac{\log(2)}{\lambda_1 - \lambda_2} \left(
\begin{array}{c}
\lambda_1^n-\lambda_2^n \\ \lambda_1^{n-1}-\lambda_2^{n-1} 
\end{array}
\right) + \sum_{i=2}^n \frac{E_i}{\lambda_1 - \lambda_2} \left(
\begin{array}{c}
\lambda_1^{n-i+1}-\lambda_2^{n-i+1} \\ \lambda_1^{n-i}-\lambda_2^{n-i} 
\end{array}
\right)
\end{align*}
for all $n\geq 2$ as $Q_0=0$ and $Q_1=\log(2)$. Therefore, we obtain
\begin{align*}
Q_n=\frac{\log(2)}{\lambda_2 - \lambda_1}\big(\lambda_2^n-\lambda_1^n \big) + \frac{1}{\lambda_2 - \lambda_1} \sum_{i=2}^n E_i \big( \lambda_2^{n-i+1}-\lambda_1^{n-i+1}\big)
\end{align*}
for all $n\geq 2$. Since the sequence $(E_n)_{n\geq 2}$ is bounded for every $n\geq 2$ (as $\lim_{n\to \infty} E_n=0$ by virtue of~\eqref{EnEp} and \eqref{IntExp}), we derive that
\begin{align*}
\sum_{i=2}^n |E_i| \cdot |\lambda_1|^{n-i+1} \leq \frac{|\lambda_1|^n-|\lambda_1|}{|\lambda_1|-1}\cdot \sup_{2\leq m\leq n}|E_m|
\end{align*}
for all $n\geq 2$. This implies that the quantity
\begin{align*}
\frac{1}{\lambda_2 - \lambda_1} \sum_{i=2}^n E_i \cdot \lambda_1^{n-i+1}
\end{align*}
converges to a definite limit as $n\to \infty$ (note that $|\lambda_1|<1$ and $\sup_{2\leq m \leq n}|E_m|$ is finite for every $n\geq 2$). On the other hand, we have
\begin{align*}
0\leq \Big|\sum_{i=n+1}^{\infty} E_i \cdot \lambda_2^{n-i+1}\Big| \leq \frac{\lambda_2}{\lambda_2 -1} \cdot \sup_{m\geq n+1}|E_m|
\end{align*}
for all $n\geq 2$ (note that $\lambda_2>1$), which implies that
\begin{align*}
\frac{1}{\lambda_2 - \lambda_1}\sum_{i=n+1}^{\infty} E_i \cdot \lambda_2^{n-i+1}=\mathcal{O}\Big(\sup_{m\geq n+1}|E_m| \Big)=o(1)
\end{align*} 
as $n\to \infty$. Putting everything together, we arrive at
\begin{align*}
Q_n&=\frac{\lambda_2^n}{\lambda_2 - \lambda_1} \Big( \log(2) +\sum_{i=2}^{\infty} E_i \cdot \lambda_2^{-i+1} \Big)-\frac{1}{\lambda_2 - \lambda_1} \sum_{i=2}^n E_i \cdot \lambda_1^{n-i+1}\\
& +\mathcal{O}\Big(\sup_{m\geq n+1}|E_m| \Big)+ \mathcal{O}(\lambda_1^n) \\
&=\frac{\lambda_2^n}{\lambda_2 - \lambda_1} \Big( \log(2) +\sum_{i=2}^{\infty} E_i \cdot \lambda_2^{-i+1} \Big)-\frac{1}{\lambda_2 - \lambda_1} \sum_{i=2}^n E_i \cdot \lambda_1^{n-i+1}+o(1)
\end{align*}
as $n\to \infty$. We deduce that
\begin{align*}
A_n&=\Big(1+\mathcal{O}\Big(\lambda_1^n + \sup_{m\geq n+1}|E_m|\Big) \Big) \\
&\hspace*{2cm} \cdot \exp\Bigg(\frac{\lambda_2^n}{\lambda_2 - \lambda_1} \Big(\log(2) +\sum_{i=2}^{\infty} E_i \cdot \lambda_2^{-i+1} \Big) -\frac{1}{\lambda_2 - \lambda_1} \sum_{i=2}^n E_i \cdot \lambda_1^{n-i+1}\Bigg)\\
&=(1+o(1))\cdot \exp\Bigg(\frac{\lambda_2^n}{\lambda_2 - \lambda_1} \Big( \log(2) +\sum_{i=2}^{\infty} E_i \cdot \lambda_2^{-i+1} \Big) -\frac{1}{\lambda_2 - \lambda_1} \sum_{i=2}^n E_i \cdot \lambda_1^{n-i+1}\Bigg)
\end{align*}
as $n\to \infty$. Call $K_2$ the quantity
\begin{align*}
\exp\Bigg(\frac{1}{\lambda_2 - \lambda_1} \Big( \log(2) +\sum_{i=2}^{\infty} E_i \cdot \lambda_2^{-i+1} \Big) \Bigg)\,,
\end{align*}
and $K_1$ the quantity
\begin{align*}
\exp\Bigg( -\frac{1}{\lambda_2 - \lambda_1} \cdot \lim_{n\to \infty}\Bigg(\sum_{i=2}^n E_i \cdot \lambda_1^{n-i+1}\Big)\Bigg)\,.
\end{align*}
Thus,
\begin{align*}
A_n =\text{N}(f_n)=(1+o(1)) K_1 \cdot K_2^{\lambda_2^n}=(1+o(1)) K_1 \cdot K_2^{\big(\frac{1+\sqrt{5}}{2}\big)^n}
\end{align*}
as $n\to \infty$, where $K_1$ and $K_2$ can now be written as
\begin{align*}
K_2&=\exp\Bigg( \frac{1}{\sqrt{5}} \Bigg( \log(2) +\sum_{i=2}^{\infty} \Big(\frac{1+\sqrt{5}}{2} \Big)^{-i+1} \\
&\hspace*{2cm} \cdot \log\Big(1+\frac{1}{2\text{N}(f_{i-1})} -\frac{\text{N}(f_{i-2})}{2\text{N}(f_{i-1})}+\frac{1}{\text{N}(f_{i-1})\cdot \text{N}(f_{i-2})}\Big) \Bigg)\Bigg)
\end{align*}
and
\begin{align*}
K_1&=\exp\Bigg( -\frac{1}{\sqrt{5}} \cdot \lim_{n\to \infty}\Bigg(\sum_{i=2}^n \Big(\frac{1-\sqrt{5}}{2} \Big)^{n-i+1}\\
& \hspace*{2cm} \cdot \log\Big(1+\frac{1}{2\text{N}(f_{i-1})} -\frac{\text{N}(f_{i-2})}{2\text{N}(f_{i-1})}+\frac{1}{\text{N}(f_{i-1})\cdot \text{N}(f_{i-2})}\Big) \Bigg)\Bigg)\,.
\end{align*}
By (numerically) evaluating $K_1$ and $K_2$, we obtain that the number of nonisomorphic leaf-induced subtrees of the leaf-Fibonacci tree $f_n$ is asymptotically
\begin{align*}
1.00001887227319\cdots  (1.48369689570172\ldots)^{\big(\frac{1+\sqrt{5}}{2}\big)^n}
\end{align*}
as $n\to \infty$. This completes the proof of the theorem.
\end{proof}

\medskip
This asymptotic formula can also be written in terms of the Fibonacci number $F_n$: the number of leaves of $f_n$ is given by
\begin{align*}
|f_n|=F_{n+2}=\frac{1}{\sqrt{5}}\Bigg(\Big(\frac{1+\sqrt{5}}{2} \Big)^{2+n}-\Big(\frac{1-\sqrt{5}}{2} \Big)^{2+n} \Bigg)
\end{align*}
for every $n$; so we deduce that 
\begin{align*}
\frac{10}{5+3\sqrt{5}}\cdot |f_n| \thicksim \Big(\frac{1+\sqrt{5}}{2} \Big)^n
\end{align*}
as $n \to \infty$. This implies that
\begin{align*}
\text{N}(f_n)& \thicksim K_1 \cdot K_2^{\frac{10}{5+3\sqrt{5}}\cdot |f_n|} \\
&= 1.00001887227319\cdots  (1.48369689570172\ldots)^{\frac{-5+3\sqrt{5}}{2}\cdot |f_n|}
\end{align*}
as $n \to \infty$. 

\section*{Acknowledgments}
The author thanks Stephan Wagner for a suggestion in the proof of the asymptotic formula.


\begin{thebibliography}{10}

\bibitem{adelsonalgorithmfor}
Georgii~M.~Adel'son-Vel'skii and Evgenii~M.~Landis.
\newblock An algorithm for the organization of information.
\newblock {\em Doklady Akademii Nauk SSSR}, 146(2):263--266, 1962.

\bibitem{allman2004mathematical}
Elizabeth~S.~Allman and John~A.~Rhodes.
\newblock Mathematical models in biology: an introduction.
\newblock {\em Cambridge University Press, Cambridge}, 2004.

\bibitem{bergeron1998combinatorial}
Fran\c{c}ois~Bergeron, Gilbert~Labelle, and Pierre~Leroux. 
\newblock Combinatorial species and tree-like structures.
\newblock {\em Cambridge University Press, Cambridge}, 1998.

\bibitem{czabarka2018inducibility}
{\'E}va~Czabarka, Audace~A.V.~Dossou-Olory, L{\'a}szl{\'o}~ A.~Sz{\'e}kely and Stephan Wagner.
\newblock Inducibility of d-ary trees.
\newblock Preprint, \url{https://arxiv.org/abs/1802.03817}, 2018.

\bibitem{dossou2018inducibility}
Audace~A.~V.~Dossou-Olory and Stephan~Wagner.
\newblock Inducibility of topological trees.
\newblock Accepted for publication in {\em Quaestiones Mathematicae}, 2018. \url{https://www.tandfonline.com/doi/abs/10.2989/16073606.2018.1497725}

\bibitem{leafSubtree}
Audace~A.~V.~Dossou-Olory and Stephan~Wagner.
\newblock On the number of leaf-induced subtrees of a topological tree.
\newblock To be submitted, 2018.

\bibitem{grimaldi1991properties}
Ralph~P.~Grimaldi.
\newblock Properties of Fibonacci trees.
\newblock {\em Congressus Numerantium}, 84:21--32, 1991.

\bibitem{horibe1982entropy}
Yasuichi~Horibe.
\newblock An entropy view of Fibonacci trees.
\newblock {\em The Fibonacci Quarterly}, 20(2):168--178, 1982.

\bibitem{horibe1983notes}
Yasuichi~Horibe.
\newblock Notes on Fibonacci trees and their optimality.
\newblock {\em The Fibonacci Quarterly}, 21(2):118--128, 1983.

\bibitem{knuth1997art}
Donald~E.~Knuth.
\newblock The art of computer programming (Volume 3): sorting and searching (second edition).
\newblock {\em Addison-Wesley Series in Computer Science and Information Processing}, 791 pp., 1998.

\bibitem{semple2003phylogenetics}
Charles~Semple and Mike~Steel.
\newblock Phylogenetics.
\newblock Oxford University Press, Oxford, 2003.

\bibitem{oeis}
Neil~J.~A. Sloane.
\newblock The On-Line Encyclopedia of Integer Sequences.
\newblock Published electronically at \url{https://oeis.org}, 2018.

\bibitem{stevens1974patterns}
Peter~S.~Stevens.
\newblock Patterns in nature.
\newblock {\em Little, Brown and Company}, United States of America, 240 pp., 1974.

\bibitem{wagner2007fibonacci}
Stephan~Wagner.
\newblock The Fibonacci number of Fibonacci trees and a related family of polynomial recurrence systems.
\newblock {\em The Fibonacci Quarterly}, 45(3):247--253, 2007.

\end{thebibliography}
\end{document}